\documentclass[graybox]{svmult}

\smartqed
\usepackage{mathptmx}       % selects Times Roman as basic font
\usepackage{helvet}         % selects Helvetica as sans-serif font
\usepackage{courier}        % selects Courier as typewriter font
\usepackage{type1cm}        % activate if the above 3 fonts are
\usepackage{graphicx}       % standard LaTeX graphics tool
\usepackage{array,colortbl}
\usepackage{amsmath,amsfonts,amssymb,bm} % no amsthm, Springer defines Theorem, Lemma, etc themselves

\newcommand{\R}{\mathbb{R}} % reals
\newcommand{\N}{\mathbb{N}} % natural numbers {1, 2, ...}
\spnewtheorem{assumption}{Assumption}{\bfseries}{}

\usepackage{microtype} % good font tricks

\usepackage[colorlinks=true,linkcolor=black,citecolor=black,urlcolor=black]{hyperref}
\usepackage{bookmark}
\makeatletter % to avoid hyperref warnings:
  \providecommand*{\toclevel@author}{999}
  \providecommand*{\toclevel@title}{0}
\makeatother

\begin{document}

\title*{A Note on the Importance of Weak Convergence Rates for SPDE Approximations in Multilevel Monte Carlo Schemes}
\titlerunning{Weak Convergence of SPDE Approximations and MLMC} 
% for an abbreviated version of
% your contribution title if the original one is too long
\author{Annika Lang}
% Use \authorrunning{Short Title} for an abbreviated version of
% your contribution title if the original one is too long
\institute{
Annika Lang 
\at Department of Mathematical Sciences, Chalmers University of Technology \& University of Gothenburg,
SE--412 96 G\"oteborg, Sweden,
\email{annika.lang@chalmers.se}
}

\maketitle

\abstract{It is a well-known rule of thumb that approximations of stochastic partial differential equations have essentially twice the order of weak convergence compared to the corresponding order of strong convergence. This is already known for many approximations of stochastic (ordinary) differential equations while it is recent research for stochastic partial differential equations. In this note it is shown how the availability of weak convergence results influences the number of samples in multilevel Monte Carlo schemes and therefore reduces the computational complexity of these schemes for a given accuracy of the approximations.
}

%%%%%%%%%%%%%%%%%%%%%%%%%%%%%%%%%%%%%%%%%%%%%%%%%
\section{Introduction}\label{sec:Intro}
%%%%%%%%%%%%%%%%%%%%%%%%%%%%%%%%%%%%%%%%%%%%%%%%%

Since the publication of Giles' articles about multilevel Monte Carlo methods~\cite{G06,G08}, which applied an earlier idea of Heinrich~\cite{H01} to stochastic differential equations, an enormous amount of literature on the application of multilevel Monte Carlo schemes to various applications has been published. For an overview of the state of the art in the area, the reader is referred to the scientific program and the proceedings of MCQMC14 in Leuven.

This note is intended to show the consequences of the availability of different types of convergence results for stochastic partial differential equations of It\^o type (SPDEs for short in what follows). Here we consider so called strong and weak convergence rates, where a sequence of approximations $(Y_\ell, \ell \in \N_0)$ of a $H$-valued random variable~$Y$ converges strongly (also called in mean square) to~$Y$ if
  \begin{equation*}
   \lim_{\ell \rightarrow +\infty} \mathbb{E}[\|Y - Y_\ell\|_H^2]^{1/2} = 0.
  \end{equation*}
In the context of this note, $H$ denotes a separable Hilbert space. The sequence is said to converge weakly to~$Y$ if
  \begin{equation*}
    \lim_{\ell \rightarrow +\infty} |\mathbb{E}[\varphi(Y)] - \mathbb{E}[\varphi(Y_\ell)]| = 0
  \end{equation*}
for $\varphi$ in an appropriately chosen class of functionals that depends in general on the treated problem.
While strong convergence results for approximations of many SPDEs are already well-known, corresponding orders of weak convergence that are better than the strong ones are just rarely available. For an overview on the existing literature on weak convergence, the reader is referred to~\cite{AKL15,JK15} and the literature therein. The necessity to do further research in this area is besides other motivations also due to the efficiency of multilevel Monte Carlo approximations, which is the content of this note. By a rule of thumb one expects the order of weak convergence to be twice the strong one for SPDEs. This is shown under certain smoothness assumptions on the SPDE and its approximation in~\cite{AKL15}. We use the SPDE from~\cite{AKL15} and its approximations with the desired strong and weak convergence rates to show that the additional knowledge of better weak than strong convergence rates changes the choices of the number of samples per level in a multilevel Monte Carlo approximation according to the theory. Since, for a given accuracy, the number of samples reduces with the availability of weak rates, the overall computational work decreases. Computing numbers, we shall see in the end that for high dimensional problems and low regularity of the original SPDE the work using only strong approximation results is essentially the squared work using also weak approximation rates. In other words the order of the complexity of the work in terms of accuracy decreases essentially by a factor of~$2$, when weak convergence rates are available. The intention of this note is to point out this important fact by writing down the resulting numbers explicitly. First simulation results are presented in the end for a stochastic heat equation in one dimension driven by additive space-time white noise, which, to the best of my knowledge, are the first simulation results of that type in the literature. The obtained results confirm the theory.

This work is organized as follows: In Section~\ref{sec:MLMC} the multilevel Monte Carlo method is recalled including results for the approximation of Hilbert-space-valued random variables on arbitrary refinements. SPDEs and their approximations are introduced in Section~\ref{sec:SPDEs} and results for strong and weak errors from~\cite{AKL15} are summarized. The results from Sections~\ref{sec:MLMC} and~\ref{sec:SPDEs} are combined in Section~\ref{sec:SPDE-MLMC} to a multilevel Monte Carlo scheme for SPDEs and the consequences of the knowledge of weak convergence rates are outlined. Finally, the theory is confirmed by simulations in Section~\ref{sec:Simulations}.

\section{Multilevel Monte Carlo for Random Variables}\label{sec:MLMC}

In this section we recall and improve a convergence and a work versus accuracy result for the multilevel Monte Carlo estimator of a Hilbert-space-valued random variable from~\cite{BL12_2}. This is used to calculate errors and computational work for the approximation of stochastic partial differential equations in Section~\ref{sec:SPDE-MLMC}. A multilevel Monte Carlo method for (more general) Banach-space-valued random variables has been introduced in~\cite{H01}, where the author derives bounds on the error for given work. Here, we do the contrary and bound the overall work for a given accuracy.

We start with a lemma on the convergence in the number of samples of a Monte Carlo estimator. Therefore let $(\Omega, \mathcal{A}, P)$ be a probability space and let $Y$ be a random variable with values in a Hilbert space~$(B,(\cdot,\cdot)_B)$ and $(\hat{Y}^i, i \in \N)$ be a sequence of independent, identically distributed copies of~$Y$. Then the strong law of large numbers states that the \emph{Monte Carlo estimator $E_N[Y]$} defined by
\begin{equation*}
 E_N[Y] := \frac{1}{N} \sum_{i=1}^N \hat{Y}^i
\end{equation*}
converges $P$-almost surely to $\mathbb{E}[Y]$ for $N \rightarrow +\infty$. In the following lemma we see that it also converges in mean square to $\mathbb{E}[Y]$ if $Y$ is square integrable, i.e., $Y \in L^2(\Omega;B)$ with
\begin{equation*}
L^2(\Omega;B)
:=
\left\{
v:\Omega \rightarrow B, \, v \text{ strongly measurable}, \,
\| v \|_{L^2(\Omega;B)} < +\infty
\right\} ,
\end{equation*}
where
\begin{equation*}
\| v \|_{L^2(\Omega;B)}
:=
\mathbb{E} [ \|v\|_B^2]^{1/2} .
\end{equation*}
In contrast to the almost sure convergence of $E_N[Y]$ derived from the strong law of large numbers, a convergence rate in mean square can be deduced from the following lemma in terms of the number of samples $N \in \N$.
\begin{lemma}\label{lem:MCerr}
For any $N\in \N$ 
and for $Y\in L^2(\Omega; B)$, it holds that
\begin{equation*}
\| \mathbb{E}[Y] - E_N[Y] \|_{L^2(\Omega;B)}
  = \frac{1}{\sqrt{N}} \, {\mathsf{Var}}[Y]^{1/2}
  \leq \frac{1}{\sqrt{N}} \, \| Y \|_{L^2(\Omega;B)}.
\end{equation*}
\end{lemma}
The lemma is proven in, e.g., \cite[Lemma~4.1]{BLS13}. It shows that the sequence of so-called \emph{Monte Carlo estimators} $(E_N[Y], N \in \N)$ converges with rate $\operatorname{O}(N^{-1/2})$ in mean square to the expectation of~$Y$.

Next let us assume that $(Y_\ell, \ell \in \N_0)$ is a sequence of approximations of~$Y$, e.g., $Y_\ell \in V_\ell$, where $(V_\ell, \ell \in \N_0)$ is a sequence of finite dimensional subspaces of~$B$. For given $L \in \N_0$, it holds that
\begin{equation*}
 Y_L = Y_0 + \sum_{\ell = 1}^L (Y_\ell - Y_{\ell - 1})
\end{equation*}
and due to the linearity of the expectation that
\begin{equation*}
  \mathbb{E}[Y_L] = \mathbb{E}[Y_0] + \sum_{\ell = 1}^L \mathbb{E}[Y_\ell - Y_{\ell - 1}].
\end{equation*}
A possible way to approximate $\mathbb{E}[Y_L]$ is to approximate $\mathbb{E}[Y_\ell - Y_{\ell - 1}]$ with the corresponding Monte Carlo estimator $E_{N_\ell}[Y_\ell - Y_{\ell - 1}]$ with a number of independent samples~$N_\ell$ depending on the level~$\ell$. We set
\begin{equation*}
 E^L[Y_L]
    := E_{N_0}[Y_0] + \sum_{\ell=1}^L E_{N_\ell}[Y_\ell - Y_{\ell - 1}]
\end{equation*}
and call $E^L[Y_L]$ the \emph{multilevel Monte Carlo estimator of $\mathbb{E}[Y_L]$}. The following lemma gives convergence results for the estimator depending on the order of weak convergence of $(Y_\ell, \ell \in \N_0)$ to~$Y$ and the convergence of the variance of $(Y_\ell - Y_{\ell - 1},\ell \in \N)$. If neither estimates on weak convergence rates nor on the convergence of the variances are available, one can use --- the in general slower --- strong convergence rates.

\begin{lemma}\label{lem:MLMCerr}%[{\cite[Lemma~2.2]{BL12_2}}]
 Let $Y \in L^2(\Omega;B)$ and let $(Y_\ell, \ell \in \N_0)$ be a sequence in $L^2(\Omega;B)$, then, for $L \in \N_0$, it holds that
\begin{align*}
 \|\mathbb{E}[Y] & - E^L[Y_L]\|_{L^2(\Omega;B)}\\
    & \le \| \mathbb{E}[Y - Y_L]\|_B
	  + \|\mathbb{E}[Y_L] - E^L[Y_L]\|_{L^2(\Omega;B)}\\
    & = \| \mathbb{E}[Y - Y_L]\|_B
	  + \left(
	  N_0^{-1} {\mathsf{Var}}[Y_0]
	  + \sum_{\ell=1}^L N_\ell^{-1} {\mathsf{Var}}[Y_\ell - Y_{\ell -1}]
	  \right)^{1/2}\\
    & \le \|Y - Y_L\|_{L^2(\Omega;B)}
	  + \left(
	    2 \sum_{\ell=0}^L N_\ell^{-1} (\|Y - Y_\ell\|_{L^2(\Omega;B)}^2
		      + \|Y - Y_{\ell-1}\|_{L^2(\Omega;B)}^2)
	  \right)^{1/2},
\end{align*}
where $Y_{-1} := 0$.
\end{lemma}

\begin{proof}
 This is essentially \cite[Lemma~2.2]{BL12_2} except that the square root is kept outside the sum. Therefore it remains to show the property of the multilevel Monte Carlo estimator that
 \begin{equation*}
  \|\mathbb{E}[Y_L] - E^L[Y_L]\|_{L^2(\Omega;B)}^2
    = N_0^{-1} {\mathsf{Var}}[Y_0]
	  + \sum_{\ell=1}^L N_\ell^{-1} {\mathsf{Var}}[Y_\ell - Y_{\ell -1}].
 \end{equation*}
 To prove this we first observe that
 \begin{equation*}
  \mathbb{E}[Y_L] - E^L[Y_L]
    = \mathbb{E}[Y_0] - E_{N_0}[Y_0]
      + \sum_{\ell=1}^L (\mathbb{E}[Y_\ell - Y_{\ell-1}] - E_{N_\ell}[Y_\ell - Y_{\ell-1}])
 \end{equation*}
 and that all summands are independent, centered random variables by the construction of the multilevel Monte Carlo estimator. Thus \cite[Proposition~1.12]{DPZ92} implies that 
 \begin{align*}
  \mathbb{E}[\|\mathbb{E}[Y_L] & - E^L[Y_L]\|_B^2]\\
    & = \mathbb{E}[\|\mathbb{E}[Y_0] - E_{N_0}[Y_0]\|_B^2]
      + \sum_{\ell=1}^L \mathbb{E}[\|\mathbb{E}[Y_\ell - Y_{\ell-1}] - E_{N_\ell}[Y_\ell - Y_{\ell-1}]\|_B^2]
 \end{align*}
 and Lemma~\ref{lem:MCerr} yields the claim.
 \qed
\end{proof}

This lemma enables us to choose for a given order of weak convergence of~$(Y_\ell, \ell \in \N_0)$ and for given convergence rates of the variances of $(Y_\ell - Y_{\ell-1}, \ell \in \N)$ the number of samples~$N_\ell$ on each level $\ell \in \N_0$ such that all terms in the error estimate are equilibrated. 

The following theorem is essentially Theorem~2.3 in~\cite{BL12_2}. While it was previously formulated for a sequence of discretizations obtained by regular subdivision, i.e., $h_\ell = C 2^{-\alpha \ell}$, it is written down for general sequences of discretizations here with improved sample sizes. For completeness we include the proof. We should also remark that the convergence with basis $2$ by regular subdivision in~\cite{BL12_2} is useful and important for SPDEs since most available approximation schemes that can be implemented are obtained in that way. Nevertheless, it is also known that the refinement with respect to basis~$2$ is not optimal for multilevel Monte Carlo approximations. Therefore it makes sense to reformulate the theorem in this more general way.

\begin{theorem}\label{thm:MLMC}
 Let $(a_\ell, \ell \in \N_0)$ be a decreasing sequence of positive real numbers that converges to zero and let $(Y_\ell, \ell \in \N_0)$ converge weakly to~$Y$, i.e., there exists a constant~$C_1$ such that
\begin{equation*}
 \| \mathbb{E}[Y - Y_\ell]\|_B
    \le C_1 \, a_\ell
      %2^{-\alpha \ell},
\end{equation*}
for $\ell \in \N_0$. Furthermore assume that the variance of $(Y_\ell - Y_{\ell-1}, \ell \in \N)$ converges with order $2\eta \in [0,2]$ with respect to $(a_\ell, \ell \in \N_0)$, i.e., there exists a constant~$C_2$ such that
\begin{equation*}
 {\mathsf{Var}}[Y_\ell - Y_{\ell -1}]
    \le C_2 \, a_\ell^{2\eta},
	%2^{-2 \beta \ell},
\end{equation*}
and that ${\mathsf{Var}}[Y_0] = C_3$.
For a chosen level~$L \in \N_0$, set $N_\ell := \lceil a_L^{-2} a_\ell^{2\eta} \ell^{1+\epsilon}\rceil
  %2^{2(\alpha L - \beta \ell)} \ell^{2(1+\epsilon)}
$, $\ell = 1,\ldots,L$, $\epsilon > 0$, and $N_0 := \lceil a_L^{-2}\rceil
  %2^{2\alpha L}
$, then the error of the multilevel Monte Carlo approximation is bounded by
\begin{equation*}
\|\mathbb{E}[Y] - E^L[Y_L] \|_{L^2(\Omega;B)}
    \le
      (C_1 + (C_3 + C_2 \, \zeta(1+\epsilon))^{1/2})\, a_L, %2^{-\alpha L} 
%     =: a_L,
\end{equation*}
where $\zeta$ denotes the Riemann zeta function, i.e., $\|\mathbb{E}[Y] - E^L[Y_L]\|_{L^2(\Omega;B)}$ has the same order of convergence as $\| \mathbb{E}[Y - Y_L]\|_B$.

Assume further that the work~$\mathcal{W}_\ell^B$ of one calculation of $Y_\ell - Y_{\ell-1}$, $\ell \ge 1$, is bounded by $C_4 \, a_\ell^{-\kappa}$ for a constant~$C_4$ and $\kappa > 0$, that the work to calculate $Y_0$ is bounded by a constant~$C_5$, and that the addition of the Monte Carlo estimators costs $C_6 a_L^{-\delta}$ for some $\delta \ge 0$ and some constant~$C_6$. Then the overall work~$\mathcal{W}_L$ is bounded by
\begin{equation*}
 \mathcal{W}_L
    \lesssim
 a_L^{-2}
	\bigl(C_5 + C_4 \sum_{\ell=1}^L a_\ell^{-(\kappa - 2\eta)} \ell^{1+\epsilon} \bigr) 
	+ C_6 a_L^{-\delta}.
\end{equation*}
If furthermore $(a_\ell, \ell \in \N_0)$ decreases polynomially, i.e., there exists $a >1$ such that $a_\ell = \operatorname{O}(a^{-\ell})$, then the bound on the computational work simplifies to
\begin{equation*}
 \mathcal{W}_L
    =
      \begin{cases}
	\operatorname{O}(a_L^{-\max\{2,\delta\}}) 
	    & \text{if} \; \kappa < 2 \eta,\\
% 	\operatorname{O}(a_L^{-2}|\log(a_L)|^{3+2\epsilon}) & \text{if} \; \kappa = 2 \beta,\\
	\operatorname{O}(\max\{a_L^{-(2 + \kappa -2\eta)}L^{2+\epsilon}, a_L^{-\delta}\}) 
	    & \text{if} \; \kappa \ge 2 \eta.
      \end{cases}
\end{equation*}
\end{theorem}

\begin{proof}
 First, we calculate the error of the multilevel Monte Carlo estimator. It holds with the made assumptions that
\begin{equation*}
 N_0^{-1} {\mathsf{Var}}[Y_0]
    \leq C_3 \, a_L^2
\end{equation*}
and, for $\ell = 1,\ldots, L$, that
\begin{equation*}
 N_\ell^{-1} {\mathsf{Var}}[Y_\ell - Y_{\ell-1}]
    \le
      C_2 \, a_L^2 a_\ell^{-2\eta} \ell^{-(1+\epsilon)} \, a_\ell^{2\eta}
    =
      C_2 \, a_L^2 \ell^{-(1+\epsilon)}.
\end{equation*}
So overall we get that
\begin{equation*}
 \sum_{\ell=1}^L N_\ell^{-1} {\mathsf{Var}}[Y_\ell - Y_{\ell-1}]
    \le
      C_2 \, a_L^2 \sum_{\ell=1}^L \ell^{-(1+\epsilon)}
    \le
      C_2 \, a_L^2 \zeta(1+\epsilon),
\end{equation*}
where $\zeta$ denotes the Riemann zeta function. 
% Otherwise we have that $2^{-\beta \ell} \le 2^{-\alpha L} \ell^{-(1+\epsilon)}$ and the same estimate is true. 
To finish the calculation of the error we apply Lemma~\ref{lem:MLMCerr} and assemble all estimates to
\begin{equation*}
 \|\mathbb{E}[Y] - E^L[Y_L] \|_{L^2(\Omega;B)}
    \le
      (C_1 + (C_3 + C_2 \, \zeta(1+\epsilon))^{1/2})\, a_L.
\end{equation*}
Next we calculate the necessary work to achieve this error. The overall work consists of the work~$\mathcal{W}_\ell^B$ to compute $Y_\ell - Y_{\ell-1}$ times the number of samples~$N_\ell$ on all levels $\ell = 1, \ldots, L$, the work~$\mathcal{W}_0^B$ on level~$0$, and the addition of the Monte Carlo estimators in the end. Therefore, using the observation that $N_\ell \simeq a_L^{-2} a_\ell^{2\eta} \ell^{1+\epsilon}$, $\ell = 1,\ldots,L$, and $N_0 \simeq a_L^{-2}$ with equality if the right hand side is an integer, we obtain that
\begin{align*}
 \mathcal{W}_L
    & \le
      C_5 N_0 + C_4 \sum_{\ell=1}^L N_\ell a_\ell^{-\kappa} 
	+ C_6 a_L^{-\delta}\\
    & \lesssim
      C_5 \, a_L^{-2}
	+ C_4 \sum_{\ell=1}^L a_L^{-2} a_\ell^{2\eta} \ell^{1+\epsilon} a_\ell^{-\kappa}
	+ C_6 a_L^{-\delta}\\  
    & \le a_L^{-2}
	\bigl(C_5 + C_4 \sum_{\ell=1}^L a_\ell^{-(\kappa - 2\eta)} \ell^{1+\epsilon} \bigr) 
	+ C_6 a_L^{-\delta},
%     & \le
%       C_5 \, 2^{2\alpha L}
% 	+ C_4 \, C_\nu 2^{2\alpha L} \sum_{\ell=1}^L 2^{(\kappa + \nu-2\beta) \ell)}
\end{align*}
which proves the first claim of the theorem on the necessary work.

If $\kappa < 2 \eta$ and additionally $(a_\ell, \ell \in \N_0)$ decreases polynomially, the sum on the right hand side is absolutely convergent and therefore
\begin{equation*}
 \mathcal{W}_L
   \lesssim (C_5 + C_4 \, C) a_L^{-2}
      + C_6 a_L^{-\delta}
    = \operatorname{O}(a_L^{-\max\{2,\delta\}}).
\end{equation*}
For $\kappa \ge 2 \eta$, it holds that
\begin{align*}
  \mathcal{W}_L
   & \lesssim a_L^{-2} (C_5 + C_4 a_L^{-(\kappa - 2\eta)} L^{2+\epsilon})
      + C_6 a_L^{-\delta}\\
    & = \operatorname{O}(\max\{ a_L^{-(2 + \kappa -2\eta)} L^{2+\epsilon}, a_L^{-\delta}\}). %\qedhere
\end{align*}
This finishes the proof of the theorem.
\qed
\end{proof}

We remark that the computation of the sum over different levels of the Monte Carlo estimators does not increase the computational complexity if $Y_\ell \in V_\ell$ for all $\ell \in \N_0$ and $(V_\ell, \ell \in \N_0)$ is a sequence of nested finite dimensional subspaces of~$B$.

%%%%%%%%%%%%%%%%%%%%%%%%%%%%%%%%%%%%%%%%%%%%%%%%
\section{Approximation of Stochastic Partial Differential Equations}\label{sec:SPDEs}
%%%%%%%%%%%%%%%%%%%%%%%%%%%%%%%%%%%%%%%%%%%%%%%%

In this section we use the framework of~\cite{AKL15} and recall the setting and the results presented in that manuscript. We use the different orders of strong and weak convergence of a Galerkin method for the approximation of a stochastic parabolic evolution problem in Section~\ref{sec:SPDE-MLMC} to show that it is essential for the efficiency of multilevel Monte Carlo methods to consider also weak convergence rates and not only strong ones as was presented in~\cite{BLS13}.

Let $(H,(\cdot,\cdot)_H)$ be a separable Hilbert space with induced norm~$\|\cdot\|_H$ and $Q:H\rightarrow H$ be a self-adjoint positive semidefinite linear operator. We define the reproducing kernel Hilbert space $\mathcal{H} = Q^{1/2}(H)$ with inner product $(\cdot,\cdot)_{\mathcal{H}} = (Q^{-1/2}\cdot,Q^{-1/2}\cdot)_H$, where $Q^{-1/2}$ denotes the square root of the pseudo inverse of~$Q$ which exists due to the made assumptions. Let us denote by $L_{\text{HS}}(\mathcal{H};H)$ the space of all Hilbert--Schmidt operators from $\mathcal{H}$ to~$H$, which will be abbreviated by $L_{\text{HS}}$ in what follows. Furthermore $L(H)$ is assumed to be the space of all bounded linear operators from $H$ to $H$. Finally, let $(\Omega, \mathcal{A}, (\mathcal{F}_t)_{t\geq 0}, P)$ be a filtered probability space satisfying the ``usual conditions'' which extends the probability space already introduced in Section~\ref{sec:MLMC}. The corresponding Bochner spaces are denoted by $L^p(\Omega;H)$, $p \ge 2$, with norms given by $\|\cdot\|_{L^p(\Omega;H)} = \mathbb{E}[\|\cdot\|_H^p]^{1/p}$. In this framework we denote by $W= (W(t), t \ge 0)$ a $(\mathcal{F}_t)_{t\geq 0}$-adapted $Q$-Wiener process. Let us consider the stochastic partial differential equation
\begin{equation}\label{eq:SPDE}
 \D X(t) = (AX(t) + F(X(t))) \, \D t + \D W(t)
\end{equation}
as Hilbert-space-valued stochastic differential equation on the finite time interval~$(0,T]$, $T < + \infty$, with deterministic initial condition $X(0) = X_0$.
We pose the following assumptions on the parameters, which ensure the existence of a mild solution and some properties of the solution which are necessary for the derivation and convergence of approximation schemes.

\begin{assumption}\label{ass:SPDE}
 Assume that the parameters of~\eqref{eq:SPDE} satisfy the following:
 \begin{enumerate}
  \item Let $A$ be a negative definite, linear operator on~$H$ such that $(-A)^{-1} \in L(H)$ and $A$ is the generator of an analytic semigroup $(S(t), t \ge 0)$ on~$H$.
  \item The initial value~$X_0$ is deterministic and satisfies $(-A)^\beta X_0 \in H$ for some $\beta \in [0,1]$.
  \item The covariance operator~$Q$ satisfies $\|(-A)^{(\beta-1)/2}\|_{L_{\text{HS}}} < + \infty$ for the same~$\beta$ as above.
  \item The drift $F: H \rightarrow H$ is twice differentiable in the sense that $F \in C_b^1(H;H) \cap C_b^2(H;\dot{H}^{-1})$, where $\dot{H}^{-1}$ denotes the dual space of the domain of~$(-A)^{1/2}$. 
 \end{enumerate}
\end{assumption}

Under Assumption~\ref{ass:SPDE}, the SPDE~\eqref{eq:SPDE} has a continuous mild solution
\begin{equation}\label{eq:mild-SPDE}
 X(t)
  = S(t) X_0
    + \int_0^t S(t-s) F(X(s)) \, \D s
    + \int_0^t S(t-s) \, \D W(s)
\end{equation}
for $t \in [0,T]$, which is in $L^p(\Omega;H)$ for all $p \ge 2$ and satisfies for some constant~$C$ that
\begin{equation*}
 \sup_{t \in [0,T]} \|X(t)\|_{L^p(\Omega;H)}
  \le C(1 + \|X_0\|_H).
\end{equation*}
We approximate the mild solution by a Galerkin method in space and a semi-implicit Euler--Maruyama scheme in time, which is made precise in what follows and spares us the treatment of stability issues.
Therefore let $(V_\ell, \ell \in \N_0)$ be a nested 
family of finite dimensional subspaces of~$V$ with refinement level~$\ell \in \N_0$, refinement sizes $(h_\ell, \ell \in \N_0)$, associated $H$-orthogonal projections $P_\ell$, and norm induced by $H$. 
For $\ell \in \N_0$, the sequence $(V_\ell, \ell \in \N_0)$ is supposed to be dense in~$H$ 
in the sense that for all $\phi \in H$, it holds that
\begin{equation*}
 \lim_{\ell \rightarrow +\infty} \| \phi - P_\ell \phi \|_H = 0.
\end{equation*}
We denote the approximate operator by $A_\ell: V_\ell \rightarrow V_\ell$ and specify the necessary properties in Assumption~\ref{ass:approx} below.
Furthermore let $(\Theta^n, n \in \N_0)$ be a sequence of 
equidistant time discretizations with step sizes $\Delta t^n$, 
i.e., for $n \in \N_0$,
\begin{equation*}
  \Theta^n := \{t^n_k= \Delta t^n k,\, k=0,\ldots,N(n)\},
\end{equation*}
where $N(n)= T/\Delta t^n$, which we assume to be an integer for simplicity reasons.
We define the fully discrete semigroup approximation by $S_{\ell,n} := (I - \Delta t^n A_\ell)^{-1} P_\ell$ and assume the following:
\begin{assumption}\label{ass:approx}
 The linear operators $A_\ell: V_\ell \rightarrow V_\ell$, $\ell \in \N_0$, and the orthogonal projectors $P_\ell: H \rightarrow V_\ell$, $\ell \in \N_0$, satisfy for all $k=1,\ldots,N(n)$ that
  \begin{equation*}
  \|(-A_\ell)^\rho S_{\ell,n}^k\|_{L(H)} \le C (t_k^n)^{-\rho}
 \end{equation*}
 for $\rho \ge 0$ and
 \begin{equation*}
  \|(-A_\ell)^{-\rho} P_\ell (-A)^\rho\|_{L(H)} \le C
 \end{equation*}
 for $\rho \in [0,1/2]$
 uniformly in $\ell, n \in \N_0$. Furthermore they satisfy for all $\theta \in [0,2]$, $\rho \in [-\theta,\min\{1,2-\theta\}]$, and $k=1,\ldots,N(n)$,
 \begin{equation*}
  \|(S(t^n_k) - S_{\ell,n}^k)(-A)^{\rho/2}\|_{L(H)}
    \le C (h_\ell^\theta + (\Delta t^n)^{\theta/2}) (t^n_k)^{-(\theta+\rho)/2}.
 \end{equation*}
\end{assumption}
The fully discrete semi-implicit Euler--Maruyama approximation is then given in recursive form for $t^n_k = \Delta t^n k \in \Theta^n$ 
and for $\ell \in \N_0$ by
\begin{equation*}
 X_{\ell,n}(t^n_k) 
   := S_{\ell,n} X_{\ell,n}(t^n_{k-1}) 
    + S_{\ell,n} F(X_{\ell,n}(t^n_{k-1})) \, \Delta t^n
    + S_{\ell,n} (W(t^n_k) - W(t^n_{k-1}))
\end{equation*}
with $X_{\ell,n}(0) := P_\ell X_0$,
which may be rewritten as
\begin{equation}\label{eq:approx}
X_{\ell,n}(t_k^n) = S_{\ell,n}^k X_0 
    + \Delta t^n \sum_{j=1}^k\ S_{\ell,n}^{k-j+1} F(X_{\ell,n}(t_{j-1}^n))
    + \sum_{j=1}^k\int_{t_{j-1}^n}^{t_j^n} S_{\ell,n}^{k-j+1} \,\D W(s).
\end{equation}
We remark here that we do not approximate the noise which might cause problems in implementations. One way to treat this problem is to truncate the Karhunen--Lo\`eve expansion of the $Q$-Wiener process depending on the decay of the spectrum of~$Q$ (see \cite{BL13,BL12_AMO}).

The theory on strong convergence of the introduced approximation scheme is already developed for some time and the convergence rates are well-known and stated in the following theorem.

\begin{theorem}[Strong convergence \cite{AKL15}]\label{thm:strong-conv}%Cor. 4.7 in [AKL15v2]
 Let the stochastic evolution equation~\eqref{eq:SPDE} with mild solution~$X$ and the sequence of its approximations~$(X_{\ell,n}, \ell,n \in \N_0)$ given by~\eqref{eq:approx} satisfy Assumptions~\ref{ass:SPDE} and~\ref{ass:approx} for some $\beta \in (0,1]$. Then, for every $\gamma \in (0,\beta)$, there exists a constant~$C>0$ such that for all $\ell, n \in \N_0$,
 \begin{equation*}
  \max_{k=1,\ldots,N(n)} \|X(t_k^n) - X_{\ell,n}(t_k^n)\|_{L^{2}(\Omega;H)}
%     = \max_{k=\todo{1},\ldots,n} \mathbb{E}(\|X(t_k^n) - X_{\ell,n}(t_k^n)\|_H^{\todo{p}})^{1/\todo{p}}
    \le C( h_\ell^{\gamma} + (\Delta t^n)^{\gamma/2}).
 \end{equation*}
\end{theorem}
It should be remarked at this point that the order of strong convergence does not exceed $1/2$ although we are considering additive noise since the regularity of the parameters of the SPDE are assumed to be rough. Under smoothness assumptions the rate of strong convergence attains one for additive noise since the higher order Milstein scheme is equal to the Euler--Maruyama scheme. Nevertheless, under the made assumptions on the regularity of the initial condition~$X_0$ and the covariance operator~$Q$ of the noise, this does not happen in the considered case.

The purpose of the multilevel Monte Carlo method is to approximate expressions of the form $\mathbb{E}[\varphi(X(t))]$ efficiently, where $\varphi:H\rightarrow \R$ is a sufficiently smooth functional. Therefore weak error estimates of the form $|\mathbb{E}[\varphi(X(t^n_k))] - \mathbb{E}[\varphi(X_{\ell,n}(t^n_k))]|$ are of importance. Before we state the convergence theorem from~\cite{AKL15}, we specify the necessary properties of~$\varphi$ in the following assumption.
\begin{assumption}\label{ass:functional}
 The functional $\varphi: H \rightarrow \R$ is twice continuously Fr\'echet differentiable and there exists an integer $m \ge 2$ and a constant~$C$ such that for all $x \in H$ and $j=1,2$,
 \begin{equation*}
  \|\varphi^{(j)}(x)\|_{L^{[m]}(H;\R)}
    \le C (1 + \|x\|_H^{m-j}),
 \end{equation*}
 where $\|\varphi^{(j)}(x)\|_{L^{[m]}(H;\R)}$ is the smallest constant~$K>0$ such that for all $u_1,\ldots, u_m \in H$,
 \begin{equation*}
  |\varphi^{(j)}(x)(u_1,\ldots,u_m)|
    \le K \|u_1\|_H \cdots \|u_m\|_H.
 \end{equation*}
\end{assumption}
Combining this assumption on the functional~$\varphi$ with Assumptions~\ref{ass:SPDE} and~\ref{ass:approx} on the parameters and approximation of the SPDE, we obtain the following result, which was proven in~\cite{AKL15} using Malliavin calculus.

\begin{theorem}[Weak convergence \cite{AKL15}]\label{thm:weak-conv}%Thm. 4.4 in [AKL15v2]
 Let the stochastic evolution equation~\eqref{eq:SPDE} with mild solution~$X$ and the sequence of its approximations~$(X_{\ell,n}, \ell,n \in \N_0)$ given by~\eqref{eq:approx} satisfy Assumptions~\ref{ass:SPDE} and~\ref{ass:approx} for some $\beta \in (0,1]$. Then, for every $\varphi:H \rightarrow \R$ satisfying Assumption~\ref{ass:functional} and all $\gamma \in [0,\beta)$, there exists a constant~$C>0$ such that for all $\ell, n \in \N_0$,
 \begin{equation*}
  \max_{k=1,\ldots,N(n)} | \mathbb{E}[\varphi(X(t_k^n)) - \varphi(X_{\ell,n}(t_k^n))]|
    \le C( h_\ell^{2\gamma} + (\Delta t^n)^\gamma).
 \end{equation*}
\end{theorem}

An example that satisfies Assumptions~\ref{ass:SPDE} and~\ref{ass:approx} is presented in Section~5 of~\cite{AKL15} and consists of a (general) heat equation on a bounded, convex, and polygonal domain which is approximated with a finite element method using continuous piecewise linear functions.

\section{SPDE Multilevel Monte Carlo Approximation}\label{sec:SPDE-MLMC}

In the previous section, we considered weak error analysis for expressions of the form $\mathbb{E}[\varphi(X(t))]$, where we approximated the mild solution~$X$ of the SPDE~\eqref{eq:SPDE} with a fully discrete scheme. Unluckily, this is not yet sufficient to compute ``numbers'' since we are in general not able to compute the expectation exactly. Going back to Section~\ref{sec:MLMC}, we recall that the first approach to approximate the expected value is to do a (singlelevel) Monte Carlo approximation. This leads to the overall error given in the following corollary, which is proven similarly to~\cite[Corollary~3.6]{BL12_2} and included for completeness.

\begin{corollary} \label{cor:error_SLMC}
Let the stochastic evolution equation~\eqref{eq:SPDE} with mild solution~$X$ and the sequence of its approximations~$(X_{\ell,n}, \ell,n \in \N_0)$ given by~\eqref{eq:approx} satisfy Assumptions~\ref{ass:SPDE} and~\ref{ass:approx} for some $\beta \in (0,1]$. Then, for every $\varphi:H \rightarrow \R$ satisfying Assumption~\ref{ass:functional} and all $\gamma \in [0,\beta)$, there exists a constant~$C>0$ such that for all $\ell, n \in \N_0$, the error of the Monte Carlo approximation is bounded by
\begin{equation*}%\label{eq:thmMCerrFTP}
\max_{k=1,\ldots,N(n)} \| \mathbb{E}[\varphi(X(t_k^n))] - E_N[\varphi(X_{\ell,n}(t_k^n)))]\|_{L^2(\Omega;\R)}
 \leq 
C \Bigl(h_\ell^{2\gamma} + (\Delta t^n)^\gamma +
\frac1{\sqrt{N}}\Bigr)
\end{equation*}
for $N \in \N$.
\end{corollary}

\begin{proof}
 By the triangle inequality we obtain that
 \begin{align*}%\label{eq:thmMCerrFTP}
  \| \mathbb{E}[\varphi(X(t_k^n))] - & E_N[\varphi(X_{\ell,n}(t_k^n)))]\|_{L^2(\Omega;\R)}\\
    & \leq 
      \| \mathbb{E}[\varphi(X(t_k^n))] - \mathbb{E}[\varphi(X_{\ell,n}(t_k^n)))]\|_{L^2(\Omega;\R)}\\
	& \qquad + \|\mathbb{E}[\varphi(X_{\ell,n}(t_k^n)))]- E_N[\varphi(X_{\ell,n}(t_k^n)))]\|_{L^2(\Omega;\R)}.
% C \Bigl(h_\ell^{2\gamma} + (\Delta t^n)^\gamma +
% \frac1{\sqrt{N}}\Bigr)
 \end{align*}
 The first term is bounded by the weak error in Theorem~\ref{thm:weak-conv} while the second one is the Monte Carlo error in Lemma~\ref{lem:MCerr}. Putting these two estimates together yields the claim.
 \qed
\end{proof}

The errors are all converging with the same speed if we couple $\ell$ and $n$ such that $h_\ell^2 \simeq \Delta t^n$ as well as the number of Monte Carlo samples~$N_\ell$ for $\ell \in \N_0$ by $N_\ell \simeq h_\ell^{-4\gamma}$.
This implies for the overall work that
\begin{equation*}
 \mathcal{W}_\ell 
    = \mathcal{W}_\ell^H \cdot \mathcal{W}_\ell^T \cdot \mathcal{W}_\ell^{\text{MC}}
    = \operatorname{O}(h_\ell^{-d} (\Delta t^n)^{-1} N_\ell)
    = \operatorname{O}(h_\ell^{-(d+2+4\gamma)})
    ,
\end{equation*}
where we assumed that the computational work in space is bounded by $\mathcal{W}_\ell^H = \operatorname{O}(h_\ell^{-d})$ for some $d \ge 0$, which refers usually to the dimension of the underlying spatial domain.

Since we have just seen that a (singlelevel) Monte Carlo simulation is rather expensive, the idea is to use a multilevel Monte Carlo approach instead which is obtained by the combination of the results of the previous two sections. In what follows we show that it is essential for the computational costs that weak convergence results are available, since the number of samples that should be chosen according to the theory depends heavily on this fact, if weak and strong convergence rates do not coincide.

Let us start under the assumption that Theorem~\ref{thm:weak-conv} (weak convergence rates) is not available. This leads to the following numbers of samples and computational work.

\begin{corollary}[Strong convergence]\label{cor:MLMC-strong-conv}
Let the stochastic evolution equation~\eqref{eq:SPDE} with mild solution~$X$ and the sequence of its approximations~$(X_{\ell,n}, \ell,n \in \N_0)$ given by~\eqref{eq:approx} satisfy Assumptions~\ref{ass:SPDE} and~\ref{ass:approx} for some $\beta \in (0,1]$.
Furthermore couple $\ell$ and $n$ such that $\Delta t^n \simeq h_\ell^2$ and for $L \in \N_0$, set $N_0 \simeq \lceil h_L^{-2\gamma}\rceil$ as well as $N_\ell \simeq \lceil h_L^{-2\gamma} h_\ell^{2\gamma} \ell^{1+\epsilon}\rceil$ for all $\ell=1,\ldots,L$ and arbitrary fixed $\epsilon > 0$.
Then, for every $\varphi:H \rightarrow \R$ satisfying Assumption~\ref{ass:functional} and all $\gamma \in [0,\beta)$, there exists a constant~$C>0$ such that for all $\ell, n \in \N_0$, the error of the multilevel Monte Carlo approximation is bounded by
\begin{equation*}%\label{eq:thmMCerrFTP}
\max_{k=1,\ldots,N(n_L)} \| \mathbb{E}[\varphi(X(t_k^{n_L}))] - E^L[\varphi(X_{L,n_L}(t_k^{n_L}))]\|_{L^2(\Omega;\R)}
 \leq 
C h_L^{\gamma},
\end{equation*}
where $n_L$ is chosen according to the coupling with~$L$.
If the work of one computation in space is bounded by $\mathcal{W}_\ell^H = \operatorname{O}(h_\ell^{-d})$ for $\ell=0,\ldots,L$ and fixed $d \ge 0$, which includes the summation of different levels, the overall work will be bounded by
 \begin{equation*}
  \mathcal{W}_L = \operatorname{O}(h_L^{-(d+2)}L^{2+\epsilon})
  .
 \end{equation*}
\end{corollary}

\begin{proof}
 We first observe that
   \begin{equation*}
  \max_{k=1,\ldots,N(n_L)} \|X(t_k^{n_L}) - X_{L,n_L}(t_k^{n_L})\|_{L^{2}(\Omega;H)}
%     = \max_{k=\todo{1},\ldots,n} \mathbb{E}(\|X(t_k^n) - X_{\ell,n}(t_k^n)\|_H^{\todo{p}})^{1/\todo{p}}
    \le C( h_L^{\gamma} + (\Delta t^n)^{\gamma/2})
    \simeq C \cdot 2 \cdot h_L^{\gamma}
 \end{equation*}
 by Theorem~\ref{thm:strong-conv} and the coupling of the space and time discretizations. Furthermore it holds that
 \begin{align*}%\label{eq:thmMCerrFTP}
    \max_{k=1,\ldots,N(n_L)} | \mathbb{E}[\varphi(X(t_k^{n_L}))] - & \mathbb{E}[\varphi(X_{L,n_L}(t_k^{n_L}))]|\\
      & \leq \max_{k=1,\ldots,N(n_L)} \|\varphi(X(t_k^{n_L})) - \varphi(X_{L,n_L}(t_k^{n_L}))\|_{L^{2}(\Omega;\R)}\\
      & \leq C \max_{k=1,\ldots,N(n_L)} \|X(t_k^{n_L}) - X_{L,n_L}(t_k^{n_L})\|_{L^{2}(\Omega;H)}\\
      & \leq C h_L^{\gamma},
\end{align*}
 since $\varphi$ is assumed to be a Lipschitz functional (cf.~\cite[Proposition~3.4]{BL13}).
 Furthermore Lemma~\ref{lem:MLMCerr} implies that
 \begin{align*}
  {\mathsf{Var}}[ & \varphi(X_{\ell,n_\ell}(t)) - \varphi(X_{\ell-1,n_{\ell-1}}(t))]\\
    & \leq 2(\|\varphi(X(t)) - \varphi(X_{\ell,n_\ell}(t))\|_{L^2(\Omega;\R)}^2
		      + \|\varphi(X(t)) - \varphi(X_{\ell-1,n_{\ell-1}}(t))\|_{L^2(\Omega;\R)}^2)\\
    & \leq C h_\ell^{2\gamma}.
 \end{align*}
 Setting $a_\ell = h_\ell^\gamma$, $\eta = 1$, and the sample numbers according to Theorem~\ref{thm:MLMC}, we obtain the claim.
 \qed
\end{proof}

If the additional information of better weak convergence rates from Theorem~\ref{thm:weak-conv} is available, the parameters that are plugged into Theorem~\ref{thm:MLMC} change, which leads for given accuracy to less samples and therefore to less computational work. This is made precise in the following corollary and the computations for given accuracy afterwards.

\begin{corollary}[Weak convergence]\label{cor:MLMC-weak-conv}
Let the stochastic evolution equation~\eqref{eq:SPDE} with mild solution~$X$ and the sequence of its approximations~$(X_{\ell,n}, \ell,n \in \N_0)$ given by~\eqref{eq:approx} satisfy Assumptions~\ref{ass:SPDE} and~\ref{ass:approx} for some $\beta \in (0,1]$.
Furthermore couple $\ell$ and $n$ such that $\Delta t^n \simeq h_\ell^2$ and for $L \in \N_0$, set $N_0 \simeq \lceil h_L^{-4\gamma}\rceil$ as well as $N_\ell \simeq \lceil h_L^{-4\gamma} h_\ell^{2\gamma} \ell^{1+\epsilon}\rceil$ for all $\ell=1,\ldots,L$ and arbitrary fixed $\epsilon > 0$.
Then, for every $\varphi:H \rightarrow \R$ satisfying Assumption~\ref{ass:functional} and all $\gamma \in [0,\beta)$, there exists a constant~$C>0$ such that for all $\ell, n \in \N_0$, the error of the multilevel Monte Carlo approximation is bounded by
\begin{equation*}%\label{eq:thmMCerrFTP}
\max_{k=1,\ldots,N(n_L)} \| \mathbb{E}[\varphi(X(t_k^{n_L}))] - E^L[\varphi(X_{L,n_L}(t_k^{n_L}))]\|_{L^2(\Omega;\R)}
 \leq 
C h_L^{2\gamma},
\end{equation*}
where $n_L$ is chosen according to the coupling with~$L$.
If the work of one computation in space is bounded by $\mathcal{W}_\ell^H = \operatorname{O}(h_\ell^{-d})$ for $\ell=0,\ldots,L$ and fixed $d \ge 0$, which includes the summation of different levels, the overall work will be bounded by
 \begin{equation*}
  \mathcal{W}_L = \operatorname{O}(h_L^{-(d+2+2\gamma)}L^{2+\epsilon})
  .
 \end{equation*}
\end{corollary}

\begin{proof}
 The proof is the same as for Corollary~\ref{cor:MLMC-strong-conv} except that we obtain
 \begin{equation*}
      \max_{k=1,\ldots,N(n_L)} | \mathbb{E}[\varphi(X(t_k^{n_L}))] - \mathbb{E}[\varphi(X_{L,n_L}(t_k^{n_L}))]|
      \leq C h_L^{2\gamma}
 \end{equation*}
 directly from Theorem~\ref{thm:weak-conv} and therefore set $a_\ell = h_\ell^{2\gamma}$, $\eta = 1/2$, and the sample numbers according to these choices in Theorem~\ref{thm:MLMC}.
 \qed
\end{proof}

If we take regular subdivisions of the grids, i.e., we set, up to a constant, $h_\ell := 2^{-\ell}$ for $\ell \in \N_0$ and rescale both corollaries such that the convergence rates are the same, i.e., the errors are bounded by $\operatorname{O}(h_\ell^{2\gamma})$, we obtain that for a given accuracy~$\epsilon_L$ on level~$L \in \N$, Corollary~\ref{cor:MLMC-strong-conv} leads to computational work 
\begin{equation*}
 \mathcal{W}_L 
    = \operatorname{O}\left(2^{2+\epsilon}\frac{2+\epsilon}{2\gamma} 
	\epsilon_L^{-(d+2)/\gamma}|\log_2 \epsilon_L|\right)
\end{equation*}
while the estimators in Corollary~\ref{cor:MLMC-weak-conv} can be computed in 
\begin{equation*}
 \mathcal{W}_L 
    = \operatorname{O}\left(\frac{2+\epsilon}{2\gamma}
	\epsilon_L^{-((d+2)/(2\gamma)+1)}|\log_2 \epsilon_L|\right).
\end{equation*}
Therefore the availability of weak convergence rates implies a reduction of the computational complexity of the multilevel Monte Carlo estimator which depends on the regularity~$\gamma$ and~$d$ referring to the dimension of the problem in space. For large~$d$, the work using strong convergence rates is essentially the squared work that is needed with the knowledge of weak rates. Additionally, for all $d\ge 0$, the rates are better and especially in dimension $d=1$ we obtain $\epsilon_L^{3/(2\gamma)+1}$ for the weak rates versus $\epsilon_L^{3/\gamma}$, where $\gamma \in (0,1)$. Nevertheless, one should also mention that Corollary~\ref{cor:MLMC-strong-conv} already reduces the work for $4\gamma > d+2$ compared to a (singlelevel) Monte Carlo approximation according to weak convergence rates. The results are put together in Table~\ref{tab:WorkMCvsMLMC} for a quick overview.
\begin{table}%[t]
\centering
\begin{tabular}{c|c|c|c}
 & Monte Carlo & MLMC with strong conv.\ & MLMC with weak conv.\\
 \hline
 general
 & $\epsilon_L^{-((d+2)/(2\gamma) + 2)}$
 & $2^{2+\epsilon}\frac{2+\epsilon}{2\gamma} \epsilon_L^{-(d+2)/\gamma}|\log_2 \epsilon_L|$
 & $\frac{2+\epsilon}{2\gamma} \epsilon_L^{-((d+2)/(2\gamma)+1)}|\log_2 \epsilon_L|$\\
 \hline
 $\gamma = 1$, omitting const.
 & $\epsilon_L^{-(d/2 + 3)}$
 & $\epsilon_L^{-(d+2)}|\log_2 \epsilon_L|$
 & $\epsilon_L^{-(d/2 + 2)}|\log_2 \epsilon_L|$\\
\end{tabular}
\caption{Computational work of different Monte Carlo type approximations for a given precision $\epsilon_L$.\label{tab:WorkMCvsMLMC}}
\end{table}

%%%%%%%%%%%%%%%%%%%%%%%%%%%%%%%%%%%%%%%%%%%%%%%%%
\section{Simulation}\label{sec:Simulations}
%%%%%%%%%%%%%%%%%%%%%%%%%%%%%%%%%%%%%%%%%%%%%%%%%

In this section simulation results of the theory of Section~\ref{sec:SPDE-MLMC} are shown, where it has to be admitted that the chosen example fits better the framework of~\cite{BLS13} since we estimate the expectation of the solution instead of the expectation of a functional of the solution. Simulations that fit the conditions of Section~\ref{sec:SPDE-MLMC} are under investigation.
Here we simulate similarly to~\cite{BL12} and~\cite{BL13} the heat equation driven by additive Wiener noise 
\begin{equation*}
 \D X(t) = \Delta X(t) \, \D t + \D W(t)
\end{equation*}
on the space interval $(0,1)$ and the time interval $[0,1]$ with initial condition $X(0,x) = \sin(\pi x)$ for $x \in (0,1)$. In contrast to previous simulations, the noise is assumed to be white in space to reduce the strong convergence rate of the scheme to (essentially)~$1/2$.
The solution to the corresponding deterministic system with $u(t) = \mathbb{E}[X(t)]$ for $t \in [0,1]$
\begin{equation*}
 \D u(t) = \Delta u(t) \, \D t
\end{equation*}
is in this case $u(t,x)= \exp(-\pi^2t)\sin(\pi x)$ for $x \in (0,1)$ and $t \in [0,1]$. 

The space discretization is done with a finite element method and the hat function basis, i.e., with the spaces $(\mathcal{S}_h, h > 0)$ of piecewise linear, continuous polynomials (see, e.g., \cite[Example~3.1]{BLS13}). 
% We use a Crank--Nicolson method for the time stepping. 
The numbers of multilevel Monte Carlo samples are calculated according to Corollary~\ref{cor:MLMC-strong-conv} and Corollary~\ref{cor:MLMC-weak-conv} with $\epsilon = 1$ to compare the convergence and complexity properties with and without the availability of weak convergence rates.
\begin{figure}%[htb]
     \centering
     \includegraphics[width=0.5\textwidth]{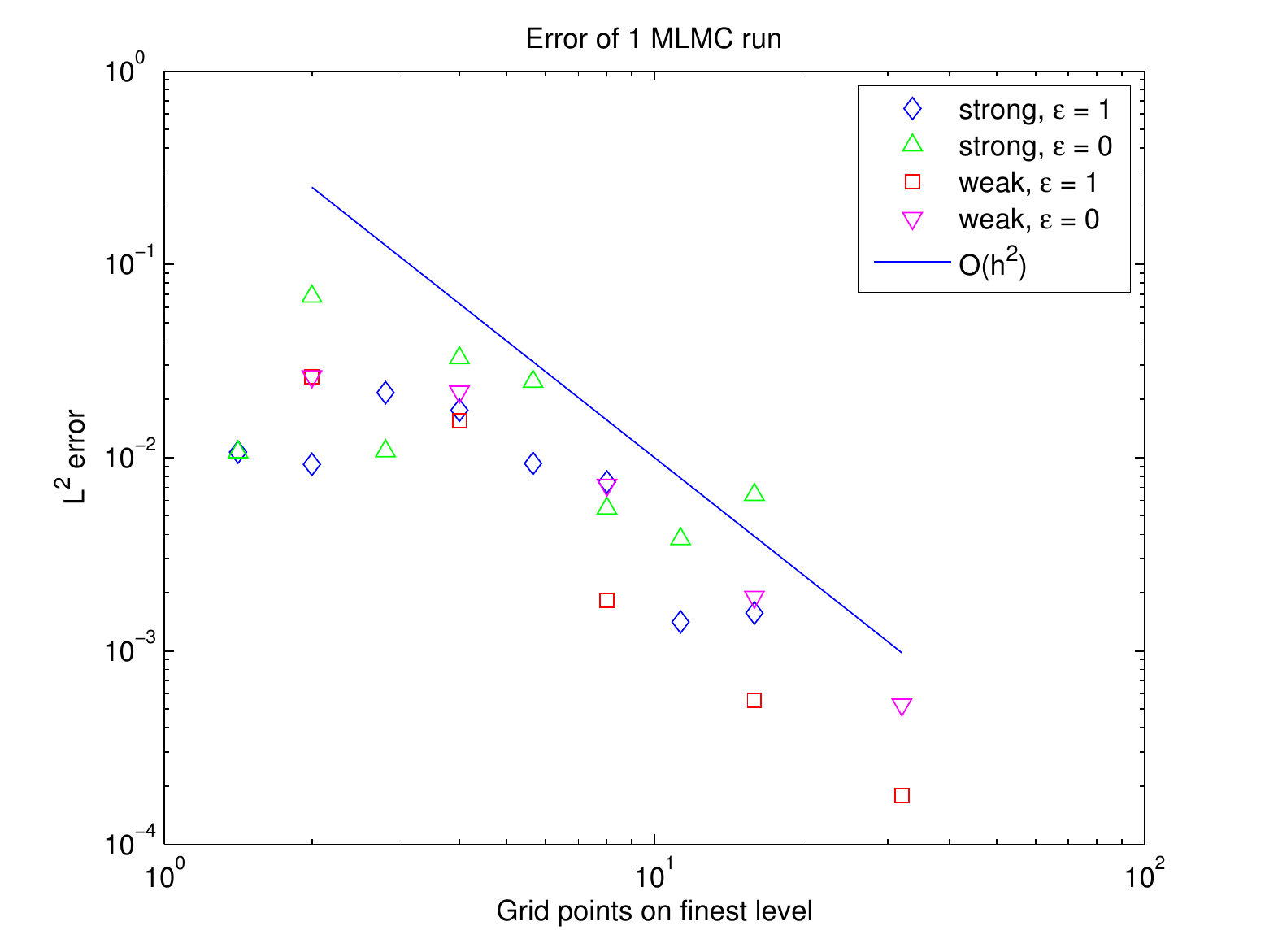}%\hfill
     \includegraphics[width=0.5\textwidth]{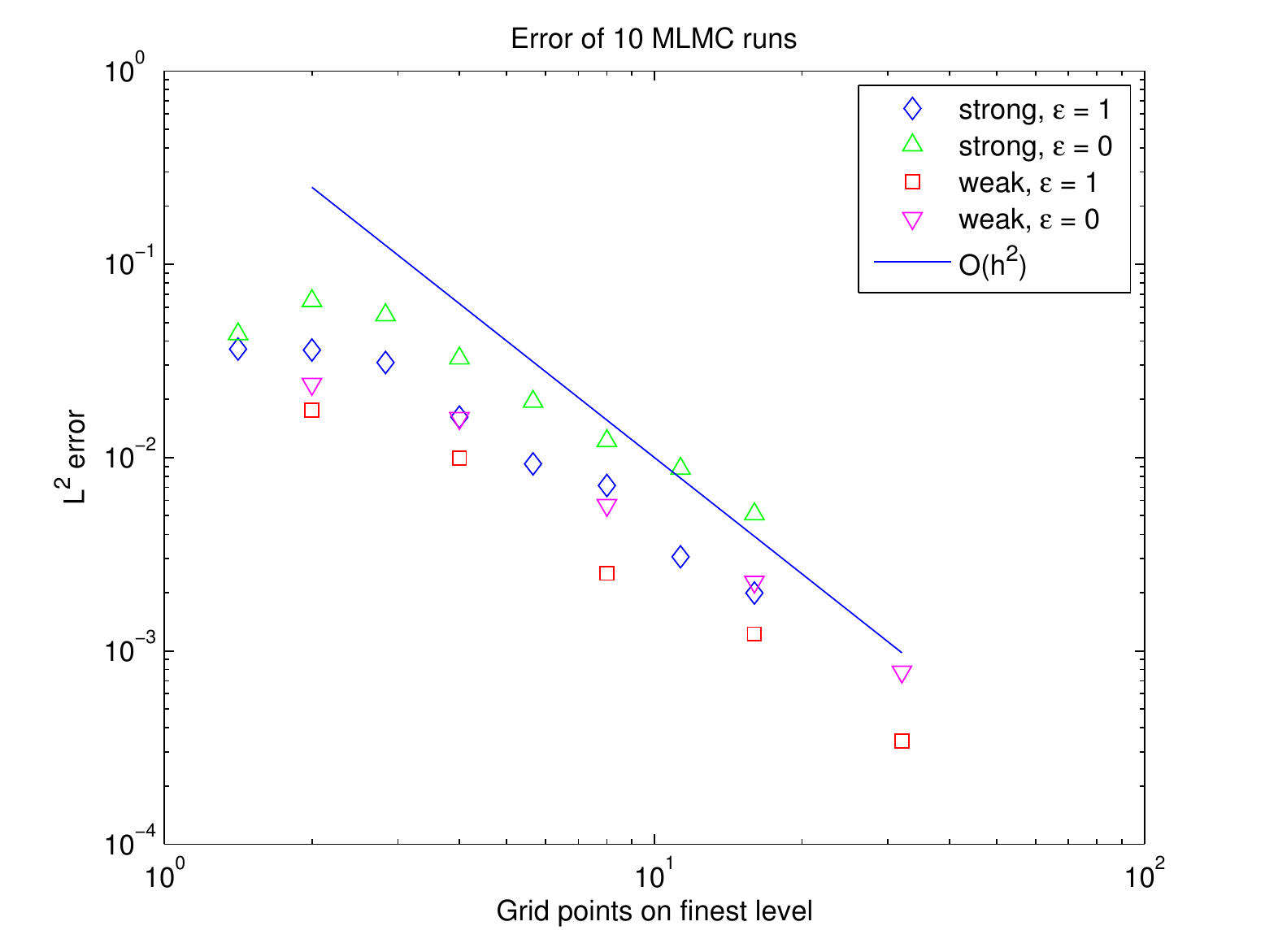}%\hfill
    \caption{Mean square error of the multilevel Monte Carlo estimator with samples chosen according to Corollary~\ref{cor:MLMC-strong-conv} and Corollary~\ref{cor:MLMC-weak-conv}.\label{fig:MLMC_error}}
\end{figure}
% %
In the left graph in Figure~\ref{fig:MLMC_error}, the multilevel Monte Carlo estimator $E^L[X_{L,2L}(1)]$ was calculated for $L = 1,\ldots,5$ for available weak convergence rates as in Corollary~\ref{cor:MLMC-weak-conv} while just for $L = 1,\ldots,4$ in the other case to finish the simulations in a reasonable time on an ordinary laptop. The plot shows the approximation of 
\begin{equation*}
 \| \mathbb{E}[X(1)] - E^L[X_{L,2L}(1)]\|_H
    = \Bigl(\int_0^{1} (\exp(-\pi^2)\sin(\pi x) - E^L[X_{L,2L}(1,x)])^2 \, dx\Bigr)^{1/2},
\end{equation*}
i.e.,
\begin{equation*}
  e_1(X_{L,2L}) := 
  \Bigl(\frac{1}{m} \sum_{k=1}^m (\exp(-\pi^2)\sin(\pi x_k) - E^L[X_{L,2L}(1,x_k)] )^2\Bigr)^{1/2}.
\end{equation*}
Here, for all levels $L=1,\ldots,5$, $m=2^5+1$ and $x_k$, $k=1,\ldots,m$, are the nodal points of the finest discretization, i.e., on level $5$ respectively $4$. The multilevel Monte Carlo estimator $E^L[X_{L,2L}]$ is calculated at these points by its basis representation for $L=1,\ldots,4$, which is equal to the linear interpolation to all grid points~$x_k$, $k=1,\ldots,m$.
One observes the convergence of one multilevel Monte Carlo estimator, i.e., the almost sure convergence of the method, which can be shown using the mean square convergence and the Borel--Cantelli lemma. In the graph on the right hand side of Figure~\ref{fig:MLMC_error}, the error is estimated by
\begin{equation*}
 e_N(X_{L,2L})
    := \Bigl( \frac{1}{N} \sum_{i=1}^N e_1(X_{L,2L}^i)^2 \Bigr)^{1/2},
\end{equation*}
where $(X_{L,2L}^i, i =1,\ldots,N)$ is a sequence of independent, identically distributed samples of~$X_{L,2L}$ and $N=10$. The simulation results confirm the theory.
\begin{figure}
  \centering
  \includegraphics[width=0.8\textwidth]{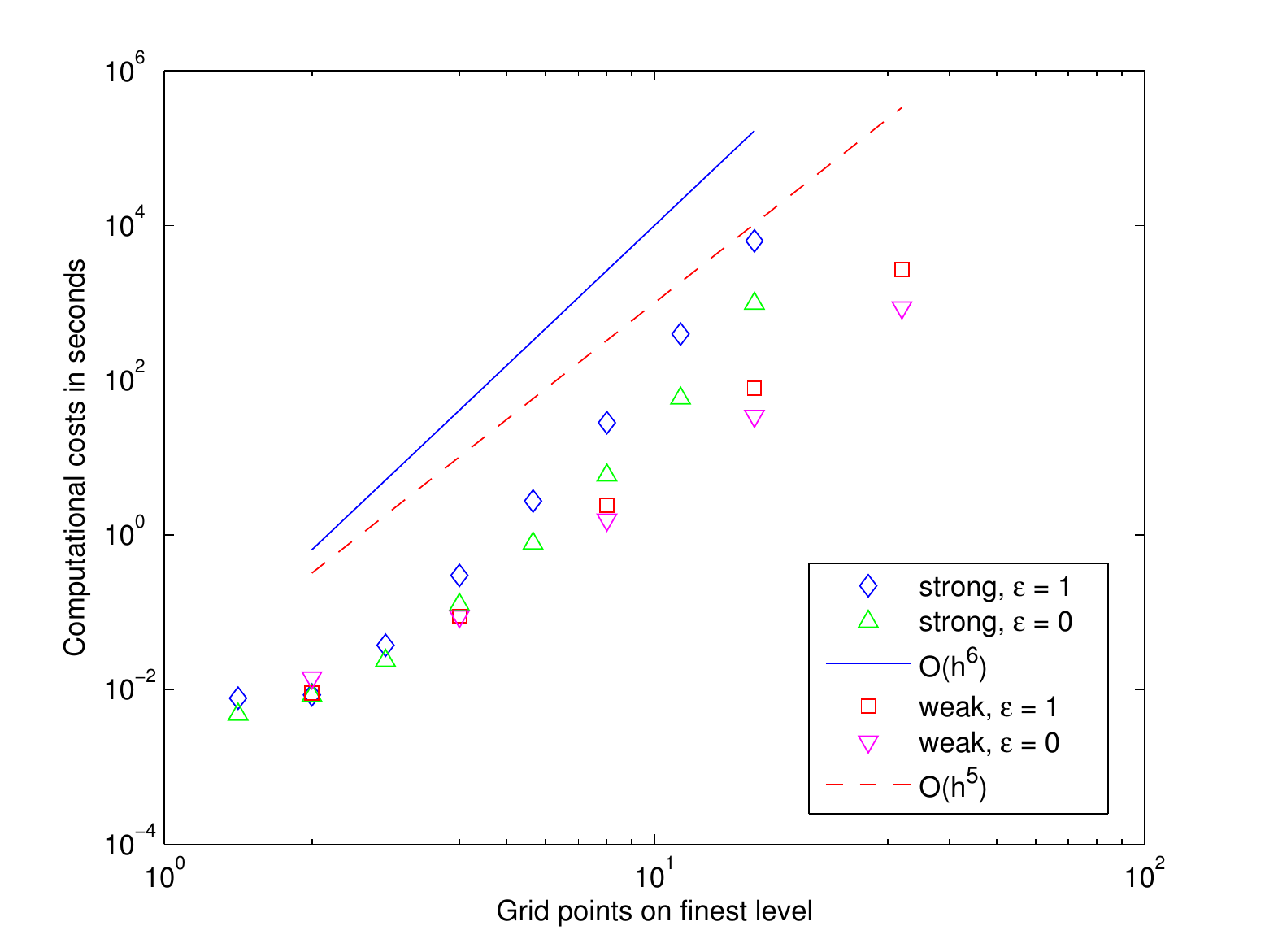}
  \caption{Computational work of the multilevel Monte Carlo estimator with samples chosen according to Corollary~\ref{cor:MLMC-strong-conv} and Corollary~\ref{cor:MLMC-weak-conv}.\label{fig:compWork}}
\end{figure}
In Figure~\ref{fig:compWork} the computational costs per level of the simulations on a laptop using matlab are shown for both frameworks. It is obvious that the computations using weak convergence rates are substantially faster. One observes especially that the computations with weak rates on level~$5$ take less time than the ones with strong rates on level~$4$. The computing times match the bounds of the computational work that were obtained in Corollary~\ref{cor:MLMC-weak-conv} and Corollary~\ref{cor:MLMC-strong-conv}.

Finally, Figure~\ref{fig:MLMC_error} and Figure~\ref{fig:compWork} include besides $\epsilon = 1$ also simulation results for the border case $\epsilon = 0$ in the choices of sample sizes per level. One observes in the left graph in Figure~\ref{fig:MLMC_error} that the variance of the errors for $\epsilon = 0$ in combination with Corollary~\ref{cor:MLMC-strong-conv} is high, which is visible in the nonalignment of the single simulation results. Furthermore the combination of Figure~\ref{fig:MLMC_error} and Figure~\ref{fig:compWork} shows that $\epsilon = 0$ combined with Corollary~\ref{cor:MLMC-weak-conv} and $\epsilon = 1$ with Corollary~\ref{cor:MLMC-strong-conv} lead to similar errors, but that the first choice of sample sizes is essentially less expensive in terms of computational complexity. Therefore the border case $\epsilon = 0$, which is not included in the theory, might be worth to consider in practice.

%%%%%%%%%%%%%%%%%%%%%%%%%%%%%%%%%%%%%%%%%%%%%%%%%%%%%%%%%%%%%%%%%%%%%%%%%%%%%%%%%%%%%%%%%%%
%%% The acknowledgements
\begin{acknowledgement}
This research was supported in part by the Knut and Alice Wallenberg foundation as well as the Swedish Research Council under Reg.~No.~621-2014-3995. The author thanks Lukas Herrmann, Andreas Petersson, and two anonymous referees for helpful comments.
\end{acknowledgement}

%%%%%%%%%%%%%%%%%%%%%%%%%%%%%%%%%%%%%%%%%%%%%%%%%%%%%%%%%%%%%%%%%%%%%%%%%%%%%%%%%%%%%%%%%%%

\bibliographystyle{spmpsci}
\bibliography{MCQMC14}

\begin{thebibliography}{10}
\providecommand{\url}[1]{{#1}}
\providecommand{\urlprefix}{URL }
\expandafter\ifx\csname urlstyle\endcsname\relax
  \providecommand{\doi}[1]{DOI~\discretionary{}{}{}#1}\else
  \providecommand{\doi}{DOI~\discretionary{}{}{}\begingroup
  \urlstyle{rm}\Url}\fi

\bibitem{AKL15}
Andersson, A., Kruse, R., Larsson, S.: Duality in refined
  {S}obolev--{M}alliavin spaces and weak approximations of {SPDE}.
\newblock Stoch. PDE: Anal. Comp.  (2015).
\newblock \doi{10.1007/s40072-015-0065-7}

\bibitem{BL12_AMO}
Barth, A., Lang, A.: {M}ilstein approximation for advection-diffusion equations
  driven by multiplicative noncontinuous martingale noises.
\newblock Appl. Math. Opt. \textbf{66}(3), 387--413 (2012).
\newblock \doi{10.1007/s00245-012-9176-y}

\bibitem{BL12_2}
Barth, A., Lang, A.: Multilevel {M}onte {C}arlo method with applications to
  stochastic partial differential equations.
\newblock Int. J. Comp. Math. \textbf{89}(18), 2479--2498 (2012).
\newblock \doi{10.1080/00207160.2012.701735}

\bibitem{BL12}
Barth, A., Lang, A.: Simulation of stochastic partial differential equations
  using finite element methods.
\newblock Stochastics \textbf{84}(2-3), 217--231 (2012).
\newblock \doi{10.1080/17442508.2010.523466}

\bibitem{BL13}
Barth, A., Lang, A.: {$L^p$} and almost sure convergence of a {M}ilstein scheme
  for stochastic partial differential equations.
\newblock Stoch. Process. Appl. \textbf{123}(5), 1563--1587 (2013).
\newblock \doi{10.1016/j.spa.2013.01.003}

\bibitem{BLS13}
Barth, A., Lang, A., Schwab, C.: Multilevel {M}onte {C}arlo method for
  parabolic stochastic partial differential equations.
\newblock BIT Num. Math. \textbf{53}(1), 3--27 (2013).
\newblock \doi{10.1007/s10543-012-0401-5}

\bibitem{DPZ92}
Da~Prato, G., Zabczyk, J.: Stochastic Equations in Infinite Dimensions,
  \emph{Encyclopedia of Mathematics and Its Applications}, vol.~44.
\newblock Cambridge: Cambridge University Press (1992).
\newblock \doi{10.1017/CBO9780511666223}

\bibitem{G06}
Giles, M.B.: Improved multilevel {M}onte {C}arlo convergence using the
  {M}ilstein scheme.
\newblock {Keller, Alexander (ed.) et al., Monte Carlo and quasi-Monte Carlo
  methods 2006. Selected papers based on the presentations at the 7th
  international conference `Monte Carlo and quasi-Monte Carlo methods in
  scientific computing', Ulm, Germany, August 14--18, 2006. Berlin: Springer.
  343--358} (2008).
\newblock \doi{10.1007/978-3-540-74496-2_20}

\bibitem{G08}
Giles, M.B.: Multilevel {M}onte {C}arlo path simulation.
\newblock Oper. Res. \textbf{56}(3), 607--617 (2008).
\newblock \doi{10.1287/opre.1070.0496}

\bibitem{H01}
Heinrich, S.: Multilevel {M}onte {C}arlo methods.
\newblock In: S.~Margenov, J.~Wasniewski, P.Y. Yalamov (eds.) Large-Scale
  Scientific Computing, Third International Conference, LSSC 2001, Sozopol,
  Bulgaria, June 6-10, 2001, Revised Papers, \emph{Lecture Notes in Computer
  Science}, vol. 2179, pp. 58--67. Springer (2001).
\newblock \doi{10.1007/3-540-45346-6_5}

\bibitem{JK15}
Jentzen, A., Kurniawan, R.: Weak convergence rates for {E}uler-type
  approximations of semilinear stochastic evolution equations with nonlinear
  diffusion coefficients (2015).
\newblock ArXiv:1501.03539 [math.PR]

\end{thebibliography}
\end{document}